\documentclass[12pt]{amsart}
\usepackage[mathscr]{eucal}
\usepackage{amsmath,amsfonts}
\usepackage[all]{xy}
\usepackage{color}
\usepackage{hyperref}

\parskip=\smallskipamount

\newtheorem{theorem}{Theorem}[section]
\newtheorem{proposition}[theorem]{Proposition}
\newtheorem{corollary}[theorem]{Corollary}
\newtheorem{lemma}[theorem]{Lemma}

\theoremstyle{definition}

\newtheorem{example}[theorem]{Example}

\newtheorem{remark}[theorem]{Remark}

\makeatletter
\@addtoreset{equation}{section}
\makeatother

\newcommand{\CC}{{\mathbb C}}
\newcommand{\NN}{{\mathbb N}}

\newcommand{\ZZ}{{\mathbb Z}}
\newcommand{\DD}{{\mathbb D}}
\newcommand{\RR}{{\mathbb R}}

\newcommand{\cB}{{\mathcal B}}
\newcommand{\cC}{{\mathcal C}}
\newcommand{\cD}{{\mathcal D}}

\newcommand{\cG}{{\mathcal G}}
\newcommand{\cH}{{\mathcal H}}

\newcommand{\cP}{{\mathcal P}}
\newcommand{\cR}{{\mathcal R}}
\newcommand{\cS}{{\mathcal S}}

\newcommand{\cW}{{\mathcal W}}

\newcommand{\fA}{{\mathfrak A}}

\newcommand{\dom}{\operatorname{Dom}}

\newcommand{\Ra}{\Rightarrow}
\newcommand{\ran}{\operatorname{Ran}}

\newcommand{\ra}{\rightarrow}

\let\phi=\varphi

\renewcommand{\ker}{\operatorname{Null}}
\newcommand{\de}{\operatorname{d}}

\newcommand{\nr}[1]{\vspace{0.1ex}\noindent\hspace*{12mm}\llap{\textup{(#1)}}}

\title[Two Generalisations Triplets of Hilbert Spaces]{A Comparison of Two Generalisations of \\
Triplets of Hilbert Spaces}
\thanks{The second named
author acknowledges financial support from
the grant PN-III-P4-PCE-2016-0823 Dynamics and Differentiable Ergodic Theory from UEFISCDI, Romania.}

\author{Petru Cojuhari}
\address{Faculty of Applied Mathematics, AGH University of Science and
Technology,\hfill\break Al.~Mickievicza 30, 30-059 Krak\'ow, Poland}
\email{cojuhari@agh.edu.pl}

\author{Aurelian Gheondea}
\address{Department of Mathematics, Bilkent University, 06800 Bilkent, Ankara,
  Turkey, \emph{and} Institutul de Matematic\u a al Academiei Rom\^ane, C.P.\
  1-764, 014700 Bucure\c sti, Rom\^ania}
\email{aurelian@fen.bilkent.edu.tr \textrm{and} A.Gheondea@imar.ro}

\begin{document}
\date{\today}

\begin{abstract} We compare the concept of triplet of closely embedded Hilbert spaces with that
of generalised triplet of Hilbert spaces in the sense of Berezanskii by showing when they coincide,
when they are different, and when starting from one of them one can naturally produce the other one that
essentially or fully coincides.
 \end{abstract}

\maketitle

\section{Introduction}

There are two basic paradigms of mathematical models in quantum physics: one due to J.~von Neumann
based on Hilbert spaces and their linear operators and the other due to P.A.M.~Dirac based on
the bra-ket duality. The two paradigms have been connected by  L.~Schwartz's theory of 
distributions \cite{Schwartz} 
and the rigged Hilbert space method, originated by I.M.~Gelfand and his school
\cite{GelfandVilenkin}, with the remarkable success of Gelfand-Maurin's Theorem 
\cite{GelfandVilenkin} and \cite{Maurin}, that turned out to be a powerful tool in analysis, partial 
differential equations, and mathematical physics.

An important more rigorous formalisation of the construction of rigged Hilbert spaces was done by
Yu.M. Berzanskii \cite{Berezanski}  
through a scale of Hilbert spaces 
and where the main step is taken by the so-called \emph{triplet of Hilbert spaces}.
More precisely,  a triplet of Hilbert spaces, denoted $\cH_{+}
\hookrightarrow \cH\hookrightarrow \cH_{-}$, means that: $\cH_{+}$, $\cH$, 
and $\cH_-$ are 
Hilbert spaces, the embeddings are continuous (bounded linear operators), the space 
$\cH_+$ is dense in $\cH$, the space $\cH$ is dense in $\cH_-$, and the space $\cH_-
$ is the dual of $\cH_+$ with respect to $\cH$, that is, $\|\phi\|_-=\sup\{ |\langle h,\phi
\rangle_{\cH}\mid \|h\|_+\leq 1\}$, for all $\phi\in \cH$. Extending these triplets on both 
sides, one may get a scale of Hilbert spaces
that yields, by an inductive and, respectively, projective limit method, a 
\emph{rigged Hilbert space} $\cS\hookrightarrow \cH\hookrightarrow \cS'$. 

In \cite{CojGh5}
we replaced the continuous embeddings by closed embeddings, that have been defined and studied 
in \cite{CojGh3}, in order 
to obtain a more flexible notion of triplet of Hilbert spaces, in the sense that we obtained
a chain of two closed embeddings with certain duality properties 
that is associated to a given Hamiltonian,
for which $0$ is yet not an eigenvalue but which may not be in the 
resolvent set. There, we obtained a model and an axiomatisation of the notion of 
triplets of closely embedded Hilbert spaces, a short study of the basic 
existence, uniqueness, a few other properties of them, and  many applications, among which a prominent role 
was played by an application to a Dirichlet problem associated to a general class of "elliptic like" partial 
differential operators. 

However, a different type of generalised triplet of Hilbert spaces was already known to Yu.M.~Berazanskii, 
see \cite{Berezanskii}, and it is the aim of this article to make a comparison of the two 
concepts. In this respect, the two concepts share some common traits, one of the most interesting being
the symmetry, see Proposition~5.3 in \cite{CojGh5} and Corollary~\ref{c:symgth}. Also, in 
Example~\ref{ex:wledoi} and Example~\ref{ex:dirichlet} we show that for the case of
two of the toy models, weighted $L^2$ spaces and Dirichlet type spaces on the unit polydisc, 
that we have used in order to derive the axiomatisation of triplets of closely embedded Hilbert spaces, these
two concepts coincide.

On the other hand, the two generalisations of triplets of Hilbert spaces are rather different in nature, 
when considered in the abstract sense. Thus, we first show in Theorem~\ref{t:genclo} that, given
a triplet of closely embedded Hilbert spaces $(\cH_+;\cH_0;\cH_-)$, the space 
$\cD=\cH_+\cap\cH_0\cap\cH_-$ is dense in each of $\cH_+$, $\cH_0$ and $\cH_-$ and that the triplet
$(\cH_+;\cH_0;\cH_-)$ becomes a generalised triplet in the sense of Berezanskii if and only if a certain 
continuity property, that can be equivalently formulated in two other different ways, holds.

For the converse problem, we first recall, in Lemma~\ref{l:be}, 
the reformulation of the original definition of Berezanskii 
of a generalised  triplet of Hilbert spaces $(\cH;\cH_0;\cH^\prime)$
in operator theoretical terms, more precisely, in terms of a contractive linear operator 
$B\colon \cH^\prime\ra\cH$ that is boundedly invertible.
Then we show in Theorem~\ref{t:gth}
that, given a generalised triplet of Hilbert spaces $(\cH;\cH_0;\cH^\prime)$, we can construct in a natural way
a triplet of closely embedded Hilbert spaces $(\cR(T);\cH_0;\cD(T^*))$
that "essentially" coincides with the triplet of generalised Hilbert spaces 
$(\cH;\cH_0;\cH^\prime)$ on the linear manifold $\cD=\cH\cap\cH_0\cap\cH^\prime$, 
that is dense in each of the spaces $\cH$, $\cH_0$, and $\cH^\prime$, 
modulo a norm equivalent with $\|\cdot\|_{\cH^\prime}$.
Finally, in Theorem~\ref{t:gtceh} characterisations of those generalised triplets of Hilbert spaces that are also 
triplets of closely embedded Hilbert spaces are obtained.

For the reader's convenience and due to the fact that the concept of triplets of closely embedded Hilbert
spaces is rather involved, we reviewed in Section~\ref{s:tcehs} all necessary constructions and the main 
results on this issue, that are needed in this article.

\section{Triplets of Closely Embedded Hilbert Spaces}
\label{s:tcehs}

A Hilbert space $\cH_+$
is called {\em closely embedded} in the Hilbert space $\cH$ if:
\begin{itemize}
\item[(ceh1)] There exists a linear manifold $\cD\subseteq
\cH_+\cap\cH$ that is dense in $\cH_+$.
\item[(ceh2)] The embedding operator $j_+$ with domain $\cD$ is closed,
as an operator $\cH_+\ra\cH$.
\end{itemize}

More precisely, axiom (ceh1) means that on $\cD$ the algebraic structures of 
$\cH_+$ and $\cH$ agree, while the meaning of the axiom (ceh2) is that the
embedding $j_+$, defined by $j_+x=x$ for all 
$x\in\cD\subseteq \cH_+$, is a closed operator when
considered as an operator from $\cH_+$ to $\cH$.
In case $\cH_+\subseteq\cH$ and the embedding operator 
$j_+\colon\cH_+\ra\cH$
is continuous, one says that $\cH_+$ is \emph{continuously embedded} in $\cH$.
The operator $A=j_+j_+^*$ is called the {\em
kernel operator} of the closely embedded Hilbert space $\cH_+$
with respect to $\cH$.

The abstract notion of closed embedding of Hilbert spaces was introduced
in \cite{CojGh3} following a generalised operator range model. In this section we first
recall two models, which are dual in a certain way, and that have been 
used in \cite{CojGh5}, and then we recall the concept of triplet 
of closely embedded Hilbert spaces and its main properties. 

\subsection{The Space $\cD(T)$.}\label{ss:dt}
We recall a first model of 
closely embedded Hilbert space generated by a closed densely defined operator. 
For the beginning, we consider a linear operator $T$ defined on a linear 
submanifold of $\cH$ and valued in $\cG$, for two Hilbert spaces $\cH$ and $\cG$,
and assume that its null space $\ker(T)$ is a closed subspace of $\cH$.
On the linear manifold $\dom(T)\ominus\ker(T)$ we consider 
the norm
\begin{equation}\label{e:normad} |x|_T:=\|Tx\|_\cG,\quad x\in\dom(T)
\ominus \ker(T),
\end{equation} and let $\cD(T)$ be the Hilbert space completion of the pre-Hilbert 
space $\dom(T)\ominus\ker(T)$ with respect to the norm $|\cdot|_T$  
associated to the inner product $(\cdot,\cdot)_T$
\begin{equation}\label{e:ipd} (x,y)_T=\langle Tx,Ty\rangle_\cG,\quad 
x,y\in \dom(T)\ominus \ker(T).
\end{equation}
We consider the operator $i_T$,
as an operator defined in $\cD(T)$ and valued in $\cH$, as follows
\begin{equation}\label{e:iote} i_T x:=x,\quad x\in\dom(i_T)=\dom(T)\ominus\ker(T).
\end{equation}
The operator $i_T$ is closed if and only if $T$ is a
closed operator, cf.\ Lemma~3.1 in \cite{CojGh5}. In addition, the construction of 
$\cD(T)$ is actually a renorming process, more precisely, the operator $T i_T$ 
admits a unique isometric extension $\widehat T\colon \cD(T)\ra \cG$, 
cf.\ Proposition~3.2 in \cite{CojGh5}. 

The most interesting case is when the operator $T$ is a closed and densely 
defined operator. Throughout this article, for two
$\cH$ and $\cG$ 
Hilbert spaces, we denote by $\cC(\cH,\cG)$ the class of all linear operators $T$ 
densely defined in $\cH$, with range in $\cG$, and closed.
The next proposition explores this case from the point of view of the
closed embedding of $\cD(T)$ in $\cH$ and that 
of the kernel operator $A=i_Ti_T^*$.

\begin{proposition}[\cite{CojGh5}, Proposition 3.3]\label{p:dete} Let $T\in\cC(\cH,\cG)$, 
for two Hilbert spaces $\cH$ and $\cG$.

\nr{a} $\cD(T)$ is closely embedded in $\cH$ and $i_T$ is the underlying closed
embedding.

\nr{b} $\ran(T^*)\subseteq \dom(i_T^*)$ and equality holds 
provided that $\ker(T)=0$.

\nr{c}  $\ran(T^*T)\subseteq \dom(i_Ti_T^*)$ and equality holds provided that 
$\ker(T)=0$. In addition, 
\begin{equation}\label{e:ititstar} (i_Ti_T^*)(T^*T)x=x,\mbox{ for all }x\in \dom(T^*T)\ominus \ker(T)
\end{equation}

\nr{d} $(i_Ti_T^*)\ran(T^*T)\subseteq\dom(T^*T)$ and equality holds provided that 
$\ker(T)=0$. In addition, 
\begin{equation}\label{e:tetestart} (T^*T)(i_Ti_T^*)u=u,\mbox{ for all }u\in 
\ran(T^*T).\end{equation}
\end{proposition}

We can view the Hilbert space $\cD(T)$ and its closed embedding $i_T$ as a 
model for the abstract definition of a closed embedding. More precisely, let 
$(\cH_+;\|\cdot\|_+)$ be a Hilbert space closely embedded in the Hilbert space
($\cH;\|\cdot\|_\cH)$ and let $j_+$ denote the 
underlying closed embedding. Since $j_+$ is one-to-one, we can define a 
linear operator $T$ with $\dom(T)=\ran(j_+)\oplus \ker(j_+^*)$, viewed as a dense
linear manifold in $\cH$, and valued in $\cH_+$,
defined by $T(x\oplus x_0)=j_+^{-1}x$, for all 
$x\in\ran(j_+)$ and $x_0\in \ker(j_+^*)$. Then $\ker(T)=\ker(j_+^*)$ and, 
for all $x\in \ran(j_+)$ we have $x=j_+u$ for a unique $u=x\in\dom(j_+)$, hence
\begin{equation*} \|x\|_+=\|Tx\|_+=|x|_T.
\end{equation*}
Thus, modulo a completion of $\dom(j_+)$ which may be different, 
the Hilbert space $(\cD(T);|\cdot|_T)$ coincides with the Hilbert space 
$(\cH_+;\|\cdot\|_+)$.

\subsection{The Space $\cR(T)$.}\label{ss:rt}
In this subsection we recall
a second model of closely embedded Hilbert spaces, based on a
construction associated to ranges of
general linear operators that was used in \cite{CojGh3}.

Let $T$ be a linear operator acting from a Hilbert space
$\cG$ to another Hilbert space $\cH$ and such that its null space $\ker(T)$ is 
closed. Introduce a pre-Hilbert space structure on $\ran(T)$ by the positive 
definite inner product $\langle \cdot,\cdot \rangle_T$ defined by
\begin{equation} \label{e:uvete}
\langle u,v \rangle_T = \langle x,y \rangle_\cG
\end{equation}
for all $u = T x$, $ v = T y$, $ x,y \in \dom (T)\ominus\ker (T)$. 
Let $\cR (T)$ be the completion of the pre-Hilbert
space $\ran (T)$ with respect to the corresponding norm $\| \cdot
\|_T$, where $\| u \|_T^{2}  = \langle u,u \rangle_T$, for $ u \in
\ran (T)$. The inner product and the norm on $\cR(T)$ are
denoted by $\langle \cdot,\cdot \rangle_T$ and, respectively, $\|\cdot\|_T$
throughout.
Consider the \emph{embedding operator} $j_T \colon\dom (j_T)
(\subseteq\cR (T)) \ra \cH$ with domain $\dom (j_T) = \ran (T)$ 
defined by
\begin{equation}\label{e:jete}
j_T u = u, \quad u \in \dom (j_T)=\ran(T).\end{equation}

Another way of viewing the definition of the Hilbert space $\cR(T)$
is by means of a certain factorisation of $T$. More precisely, letting $T$ be a linear 
operator with domain dense in the Hilbert space $\cG$, valued in the Hilbert
space $\cH$, and with closed null space, considering the Hilbert
space $\cR(T)$  and the embedding $j_T$ defined as in \eqref{e:uvete} and,
respectively, \eqref{e:jete}, there exists a unique coisometry
$U_T\in\cB(\cG,\cR(T))$, such that
$\ker(U_T)=\ker(T)$ and $T=j_TU_T$, cf.\ Lemma~2.5 in \cite{CojGh3}. 
In addition, with the notation as before,
the operator $T$ is closed if and only if the embedding operator $j_T$ is
closed, cf.\ Proposition~2.7 in \cite{CojGh3}. 

\begin{lemma}[\cite{CojGh5}, Lemma 3.8]\label{l:star}
Let $T\in\cC(\cG,\cH)$. Then $\dom(j_T^*)\supseteq \dom(T^*)$ and, if $T$ is 
one-to-one, then $\dom(j_T^*)=\dom(T^*)$.\end{lemma}
 
The definition of closely embedded Hilbert spaces is 
consistent with the model $\cR(T)$ for $T\in \cC(\cG,\cH)$, more precisely,  
if $\cH_+$ is closely embedded in $\cH$ then $\cR(j_+)=\cH_+$ and 
$\|x\|_+=\|x\|_{j_+}$.
On the other hand, the model for the abstract definition of closely embedded Hilbert 
spaces follows the results on the Hilbert space $\cR(T)$. Thus, if
$T\in\cC(\cG,\cH)$ then the Hilbert space $\cR(T)$, with its canonical
embedding $j_T$ as defined in \eqref{e:uvete} and \eqref{e:jete}, is a Hilbert
space closely embedded in $\cH$. Conversely, if $\cH_+$ is a Hilbert space closely
embedded in $\cH$, and $j_+$ denotes its canonical closed embedding, then
$\cH_+$ can be naturally viewed as the Hilbert space of type $\cR(j_+)$. In 
addition, $TT^*=j_Tj_{T}^*$, cf.\ Proposition~3.2 in \cite{CojGh3}.

The closely embedded Hilbert space $\cR(T)$ is essentially an 
operator range construction.
\begin{theorem}[\cite{CojGh3}, Theorem 2.10]\label{t:range}
Let $T\in\cC(\cG,\cH)$ be nonzero and $u\in\cH$. Then
  $u\in \ran(T)$ if and only if there exists $\mu_u\geq 0$ such that
  $|\langle u,v\rangle_\cH| \leq \mu_u \|T^*v\|_\cG$ for all
  $v\in\dom(T^*)$. Moreover, if $u\in \ran(T)$ then
\begin{equation*} \|u\|_T=\sup\bigr\{ \frac{|\langle
  u,v\rangle_\cH|}{\|T^*v\|_\cG}  \mid
  v\in\dom(T^*),\ T^*v\neq 0\bigl\},\end{equation*}
where $\|\cdot\|_T$ is the norm associated to the inner product defined
as in \eqref{e:uvete}.
\end{theorem}

The uniqueness of a closed embedded Hilbert space in terms of its kernel operator 
takes a slightly weaker form than in the case of continuous embeddings.

\begin{theorem}[\cite{CojGh3}, Theorem 3.4]\label{t:kernel} Let $\cH_+$ be a Hilbert space
closely embedded in $\cH$, with $j_+ : \cH_+ \ra
\cH$ its densely defined and closed embedding operator, and let $A
= j_+ j_+^*$ be the kernel operator associated to this embedding.
Then

\nr{a} $\ran(A^{1/2})=\dom(j_+)$ is dense in both $\cR(A^{1/2})$ and $\cH_+$.

\nr{b} For all $x\in\ran(A^{1/2})$ and all $y\in \dom(A)$ we have $\langle
x,y\rangle_\cH=\langle x, Ay\rangle_+=\langle
x,Ay\rangle_{A^{1/2}}$.

\nr{c} $\ran(A)$ is dense in both $\cR (A^{1/2})$ and $\cH_+$.

\nr{d} For any $x\in\dom(j_+)$ we have
\begin{equation*} \|x\|_+=\sup\bigl\{ \frac{|\langle
    x,y\rangle_\cH|}{\|A^{1/2}y\|_\cH} \mid y\in\dom(A^{1/2}),\ A^{1/2}y\neq
    0\bigr\}.
\end{equation*}

\nr{e} The identity operator $:\ran(A))(\subseteq\cR (A^{1/2})) \ra
\cH_+$ uniquely extends to a unitary operator $V:\cR(A^{1/2})\ra \cH_+$
such that $
V A x = j_+^* x$, for all $ x \in \dom(A)$.
 \end{theorem}

\subsection{Triplets of Closely Embedded Hilbert Spaces}\label{ss:tcehs}

By definition, $(\cH_+;\cH_0;\cH_-)$ is called a \emph{triplet of closely embedded 
Hilbert spaces} if:

\begin{itemize}
\item[(th1)] $\cH_+$ is a Hilbert space 
closely embedded in the Hilbert space
$\cH_0$, with the closed embedding denoted by $j_+$, and such that $\ran(j_+)$ is 
dense in $\cH_0$.
\item[(th2)] $\cH_0$ is closely embedded in the Hilbert space $\cH_-$, 
with the closed embedding denoted by $j_-$, and such that $\ran(j_-)$ is dense in 
$\cH_-$.
\item[(th3)] $\dom(j_+^*)\subseteq\dom(j_-)$ and for every vector
$y\in\dom(j_-)\subseteq \cH_0$ we have
\begin{equation}\label{e:dual}
 \|y\|_-=\sup\bigl\{\frac{|\langle x,y\rangle_{\cH_0}|}{\|x\|_+}\mid 
x\in\dom(j_+),\ x\neq 0\bigr\}.
\end{equation}
\end{itemize}

\begin{remark} (1) Let us first observe that, by axiom (th3),  
it follows that actually the inclusion in \eqref{e:dual} is an equality
\begin{equation}\label{e:domeq}\dom(j_+^*)=\dom(j_-).\end{equation}

(2) Let $\cH_+$, $\cH_0$, and $\cH_-$ be three Hilbert spaces. Then 
$(\cH_+;\cH_0;\cH_-)$ makes a triplet of closely embedded Hilbert spaces 
if and only the axioms (th1), (th2), and
\begin{itemize}
\item[(th3)$^\prime$] $\dom(j_+^*)=\dom(j_-)$ and $\|j_-y\|_-=\|j_+^*y\|_+$, for all 
$y\in\dom(j_-)$.
\end{itemize}
hold.

Indeed, let us assume that the axioms (th1) and (th2) hold. If (th3) holds as well, 
then, by the previous remark, $\dom(j_+^*)=\dom(j_-)$ holds. Then, 
applying Theorem~\ref{t:range} to the closed densely defined operator $j_+$ we 
obtain that $\|j_-y\|_-=\|j_+^*y\|_+$, for all $y\in\dom(j_-)$, hence the axiom 
(th3)$^\prime$ holds as well. Similarly we show that (th3)$^\prime$ implies (th3). 
\end{remark}

The concept of a triplet of closely embedded Hilbert spaces was obtained in 
\cite{CojGh5} as a consequence of a model, starting from a positive selfadjoint 
operator $H$ in a Hilbert space $\cH$ with trivial kernel, and a factorisation 
$H=T^*T$, with $T$ a closed operator densely defined in $\cH$ having trivial kernel 
and dense range in the Hilbert space $\cG$, 
and based on the spaces of type $\cD(T)$ and $\cR(T)$, see 
subsections~\ref{ss:dt} and \ref{ss:rt}. Under these assumptions, 
$(\cD(T);\cH;\cR(T^*))$ is a triplet of closely embedded Hilbert spaces 
with some additional remarkable properties.

\begin{theorem}[\cite{CojGh5}, Theorem 4.1]\label{t:model} Let $H$ be a positive selfadjoint operator in the Hilbert 
space $\cH$, with trivial null space. Let $T\in\cC(\cH,\cG)$ be such that 
$\ran(T)$ is dense in $\cG$ and $H=T^*T$. Then:
\begin{itemize}
\item[(i)] The Hilbert space $\cD(T)$ is closely embedded in $\cH$ with its closed 
embedding $i_T$ having range dense in $\cH$, and its kernel operator 
$A=i_Ti_T^*$ 
coincides with $H^{-1}$.
\item[(ii)] $\cH$ is closely embedded in the Hilbert space $\cR(T^*)$ with its closed 
embedding $j_{T^*}^{-1}$ having range dense in $\cR(T^*)$. The kernel operator 
$B=j_{T^*}^{-1}{j_{T^*}^{-1*}}$ of this closed embedding is unitary equivalent 
with $A=H^{-1}$.
\item[(iii)] The operator $i_T^*|\ran(T^*)$ 
extends uniquely to a unitary operator 
$\widetilde A$ between the Hilbert spaces $\cR(T^*)$ and $\cD(T)$. In addition,
$\widetilde A$ is the unique unitary extension of the kernel operator $A$, when 
viewed as an operator acting from $\cR(T^*)$ and valued in $\cD(T)$, as well.
\item[(iv)] The operator $H$ can be viewed as a linear operator with domain dense 
in $\cD(T)$ and dense range in $\cR(T^*)$ and, when viewed in this way, it is isometric and
it extends uniquely to a unitary operator 
$\widetilde H\colon \cD(T)\ra\cR(T^*)$, and $\widetilde H={\widetilde A}^{-1}$.
\item[(v)] Letting $V_T\in\cB(\cG,\cD_T)$ denote the unitary operator such that 
$T^{-1}=i_T V_T$ and $U_{T^*}\in\cB(\cG,\cR(T^*))$ denote the unitary 
operator such 
that $T^*=U_{T^*}j_{T^*}$, we  have $\widetilde H=U_{T^*}V_T^{-1}$.
\item[(vi)] The operator $\Theta$ defined by 
\begin{equation*}
\Theta \colon \cR(T^*)\ra\cD(T)^*,\quad (\Theta\alpha)(x):=(\widetilde A\alpha,x)_T,
\quad \alpha\in\cR(T^*),\ x\in\cD(T),
\end{equation*} provides a canonical 
identification of the Hilbert space $\cR(T^*)$ with the conjugate dual
space $\cD(T)^*$ and, for all $y\in\dom(T^*)$
\begin{equation*}
\|y\|_{T^*}=\sup\bigl\{
\frac{|\langle y,x\rangle_\cH|}{|x|_T}\mid x\in\dom(T)
\setminus\{0\}\bigr\}.
\end{equation*}
\end{itemize}
\end{theorem}

On the other hand, the properties of the triplet $(\cD(T);\cH;\cR(T^*))$ as exhibited in 
the previous theorem can be proven for any other triplet of closely embedded Hilbert 
spaces.

\begin{theorem}[\cite{CojGh5}, Theorem 5.1]\label{t:triplet} Let $(\cH_+;{\cH_0};\cH_-)$ be a triplet of closely 
embedded Hilbert spaces, and let $j_\pm$ denote the corresponding closed 
embeddings of $\cH_+$ in ${\cH_0}$ and, respectively, of ${\cH_0}$ in $\cH_-$. 
Then:

\nr{a} The kernel operator $A=j_+j_+^*$ is positive selfadjoint in ${\cH_0}$ and $0$ 
is not an eigenvalue for $A$. Also, the Hamiltonian operator $H=A^{-1}$ is a 
positive selfadjoint operator in ${\cH_0}$ for which $0$ is not an eigenvalue.

\nr{b} $\dom(j_+^*)=\dom(j_-)$, the closed embeddings $j_+$ and $j_-$ are 
simultaneously continuous or not, and the operator $V=j_+^*\colon\dom(j_+^*)
(\subseteq\cH_-)\ra\cH_+$ extends uniquely to a unitary operator 
$\widetilde V\colon \cH_-\ra\cH_+$. 

\nr{c} The kernel operator $A$ can be viewed as an operator densely defined in 
$\cH_-$ with dense range in $\cH_+$, and it is a restriction of the unitary operator 
$\widetilde V$.

\nr{d} The Hamiltonian operator $H$ can be 
viewed as an operator densely defined in $\cH_+$ with range dense in 
$\cH_-$, and it is uniquely extended to a unitary operator 
$\widetilde H\colon \cH_+\ra\cH_-$, and $\widetilde H=\widetilde{V}^{-1}$.

\nr{e} 
The operator $\Theta$ defined by $(\Theta y)(x)=\langle\widetilde Vy,x\rangle_+$, 
for all $y\in\cH_-$ and all $x\in\cH_+$ provides a unitary identification of $\cH_-$ 
with the conjugate dual space $\cH_+^*$.
\end{theorem}

The fact that there is a rather general model for triplets of closely embedded 
Hilbert spaces, as in Theorem~\ref{t:model}, can be used to prove certain 
existence and uniqueness results.

\begin{theorem}[\cite{CojGh5}, Theorem 5.2]\label{t:existence2} Assume that ${\cH_0}$ and $\cH_+$ are two 
Hilbert spaces such that $\cH_+$ is closely embedded in ${\cH_0}$, with $j_+$ denoting 
this closed embedding, and such that $\ran(j_+)$ is dense in ${\cH_0}$.
 
\nr{1} One can always extend this closed embedding to the triplet 
$(\cH_+;{\cH_0};\cR(j_+^{-1*}))$ of closely embedded Hilbert spaces.

\nr{2} Let $(\cH_+;{\cH_0};\cH_-)$ be any other extension of the closed embedding 
$j_+$ to a triplet of closely embedded Hilbert spaces, let $A=j_+j_+^*$ be its 
kernel operator, and let $j_-$ denote the closed embedding of ${\cH_0}$ in $\cH_-$. 
Then, there exists a unique unitary operator 
$\Phi_-\colon \cH_-\ra\cR(j_+^*)$ such that when restricted to $\dom(j_-)$ acts as 
the identity operator. 
\end{theorem}

Another consequence of the existence of the model is a certain 
"left-right" symmetry, which, in general, the classical triplets of Hilbert spaces 
do not share.

\begin{proposition}[\cite{CojGh5}, Proposition 5.3]\label{p:symmetry}
Let $(\cH_+;\cH_0;\cH_-)$ be a triplet of closely embedded Hilbert spaces. Then 
$(\cH_-;\cH_0;\cH_+)$ is also a triplet of closely embedded Hilbert spaces, more 
precisely:

\nr{1} If $j_+$ and $j_-$ denote the closed embeddings of $\cH_+$ in $\cH_0$ and, 
respectively, of $\cH_0$ in $\cH_-$, then $j_-^{-1}$ and $j_+^{-1}$ are the closed 
embeddings of $\cH_-$ in $\cH_0$ and, respectively, of $\cH_0$ in $\cH_+$.

\nr{2} If $H$ and $A$ denote the Hamiltonian, respectively, the kernel operator of the 
triplet $(\cH_+;\cH_0;\cH_-)$, then $A$ and $H$ are the Hamiltonian and, 
respectively, the kernel operator of the triplet $(\cH_-;\cH_0;\cH_+)$.
\end{proposition}

\section{Generalised Triplets in the Sense of Berezanskii} \label{s:gt}

In this section we
compare the notion of triplet of closely embedded Hilbert spaces with that of generalised triplet of Hilbert spaces as defined in 
\cite{Berezanskii} at page 57, more precisely, $(\cH;\cH_0;\cH^\prime)$ is called a 
\emph{generalised triplet} of Hilbert spaces if:
\begin{itemize}
\item[(gt1)] $\cD=\cH\cap\cH_0\cap\cH^\prime$ is a linear subspace dense 
in each of the Hilbert spaces $\cH$, $\cH_0$, $\cH^\prime$.
\item[(gt2)] The sesquilinear form $b(\phi,\psi)=\langle \phi,\psi\rangle_{\cH_0}$, 
$\phi,\psi\in \cD$, has the property
\begin{equation*} |b(\phi,u)|\leq \|\phi\|_{\cH^\prime}\|u\|_{\cH},\quad 
\phi,u\in\cD,
\end{equation*} and hence it can be uniquely extended to a continuous 
sesquilinear form $\cH^\prime\times\cH\ni (\phi,v)\mapsto b(\phi,v)\in \CC$.
\item[(gt3)] For each $u\in \cH$ there exists a unique vector $\phi_u\in\cH^\prime$ 
such that $\langle u,v\rangle_\cH=b(\phi_u,v)$, for all 
$v\in \cH$.
\end{itemize}

Preferable is to reformulate this definition in operator theoretical terms.

\begin{lemma}[\cite{Berezanskii}, page 58]\label{l:be}
 Let $\cH$, $\cH_0$,  and $\cH^\prime$ 
be Hilbert spaces. Then $(\cH;\cH_0;\cH^\prime)$ is a generalised triplet of Hilbert 
spaces if and only if \emph{(gt1)} holds and 
there exists $B\colon\cH^\prime\ra\cH$, a contractive and 
boundedly invertible operator, such that
\begin{equation}\label{e:beta}
\langle\phi,u\rangle_{\cH_0}=\langle B\phi,u\rangle_{\cH},\quad \phi,u\in\cD.
\end{equation} In addition, the operator $B$ is uniquely determined, subject to 
these properties.
\end{lemma}

\begin{proof} Let $(\cH;\cH_0;\cH^\prime)$ be a generalised triplet of Hilbert spaces. 
By (gt2) and the Riesz's Representation Theorem for bounded sesquilinear 
forms, there exists $B\colon\cH^\prime\ra\cH$, a contractive
operator, such that
\begin{equation}\label{e:b}
b(\phi,u)=\langle B\phi,u\rangle_{\cH},\quad \phi\in\cH^\prime,\ 
u\in\cH.
\end{equation}
By (gt3) and the Closed Graph Theorem, the operator $B$ is boundedly invertible.

Conversely, if $\cH$, $\cH_0$, and $\cH^\prime$ satisfy the axiom (gt1) and, 
in addition,  
there exists $B\colon\cH^\prime\ra\cH$, a contractive and boundedly invertible
operator such that \eqref{e:b} holds, it is easy to see that (gt2) and (gt3) hold, hence  
$(\cH;\cH_0;\cH^\prime)$ is a generalised triplet of Hilbert spaces.
\end{proof}

As a consequence, it can be shown that the concept of generalised triplet of Hilbert 
spaces has a property of symmetry similar to that of the concept of triplet of closely 
embedded Hilbert spaces.

\begin{corollary}[\cite{Berezanskii}, Theorem 1.2.10]\label{c:symgth} If $(\cH;\cH_0;\cH^\prime)$ 
is a generalised triplet of Hilbert spaces, then $(\cH^\prime;\cH_0;\cH)$ is the same.
\end{corollary}

\begin{proof} Assuming that $(\cH;\cH_0;\cH^\prime)$ is a generalised triplet of Hilbert 
spaces, by Lemma~\ref{l:be}, 
let $B\colon\cH^\prime\rightarrow\cH$ be the contractive and boundedly 
invertible linear operator such that \eqref{e:beta} holds. Then 
$C=B^*\colon\cH\ra\cH^\prime$ is a contractive and boundedly invertible linear
operator such that
\begin{equation*}\langle \phi,u\rangle_{\cH_0}=\langle\phi,Cu\rangle_{\cH^\prime},
\quad \phi,u\in\cD,
\end{equation*} hence, again by Lemma~\ref{l:be}, $(\cH^\prime;\cH_0;\cH)$ is a
generalised triplet of Hilbert spaces.
\end{proof}

\subsection{Starting with a Triplet of Closely Embedded Hilbert Space.}
We first investigate the possibility of making a generalised triplet from a triplet 
of closely embedded Hilbert spaces.

\begin{theorem}\label{t:genclo} 
Let $(\cH_+;\cH_0;\cH_-)$ be a triplet of closely
embedded Hilbert spaces and let $\cD=\cH_+\cap\cH_0\cap\cH_-$. 
Then, $(\cH_+;\cH_0;\cH_-)$ is a generalised triplet of Hilbert spaces 
if and only if one, hence all, of the following mutually equivalent 
conditions holds:

\nr{a} The sesquilinear form $(\cD;\|\cdot\|_-)\times(\cD;\|\cdot\|_+)\ni(\phi,u)
\mapsto\langle\phi,u\rangle_{\cH_0}\in\CC$ is separately continuous.

\nr{b} The sesquilinear form $(\cD;\|\cdot\|_-)\times(\cD:\|\cdot\|_+)
\ni(\phi,u)\mapsto\langle\phi,u\rangle_{\cH_0}\in\CC$ 
is jointly continuous.

\nr{c} $|\langle\phi,u\rangle_{\cH_0}|\leq \|\phi\|_{\cH_-}\|u\|_{\cH_+}$ for all 
$\phi,u\in\cD$.
\end{theorem}

\begin{proof} Let $(\cH_+;\cH_0;\cH_-)$ be a triplet of closely
embedded Hilbert spaces and let $\cD=\cH_+\cap\cH_0\cap\cH_-$.
In order to prove that the axiom (gt1) holds, we first 
prove that $\cD$ is dense in each of $\cH_+$ and $\cH_0$.
To see this, we first observe that
\begin{align*}
\dom(j_+^*j_+) & = \{u\in\dom(j_+)\mid j_+u\in \dom(j_+^*)\}\\
& = \dom(j_+)\cap \dom(j_+^*),
\mbox{ since $j_+u=u$ for all $u\in\dom(j_+)$}\\
& = \dom(j_+)\cap \dom(j_-), 
\mbox{ since $\dom(j_+^*)=\dom(j_-)$, see \eqref{e:domeq},}\\
& = \ran(j_+)\cap\ran(j_-),
\mbox{ since $j_+$ and $j_-$ are identity operators on their domains.}
\end{align*}
Thus, $\dom(j_+^*j_+)$ is a subspace of each of the spaces $\cH_+$, 
$\cH_0$, and $\cH_-$ hence, in particular, it is a subspace of $\cD$. 
On the other hand, since $\dom(j_+^*j_+)$ is a core for $j_+$, for any 
$u\in\dom(j_+)$ there exists a sequence $(u_n)_n$ of vectors in 
$\dom(j_+^*j_+)$ 
such that $\|u_n-u\|_{\cH_+}\ra 0$ and $\|u_n-u\|_{\cH_0}\ra 0$ 
as $n\ra\infty$, 
hence, since $\dom(j_+)$ is dense in $\cH_+$ and $\ran(j_+)$ is dense in 
$\cH_0$, 
it follows that $\dom(j_+^*j_+)$ in dense in each of $\cH_+$ and $\cH_0$. In 
particular, $\cD$ is dense in each of $\cH_+$ and $\cH_0$.

In order to finish proving that the axiom (gt1) holds, it remains to prove that 
$\cD$ is dense in $\cH_-$ as well. To this end, by the symmetry property as in 
Proposition~\ref{p:symmetry}, $(\cH_-;\cH_0;\cH_+)$ is a triplet of 
closely embedded
Hilbert spaces as well and hence, using the fact just proven that $\cD$, whose 
definition does not depend on the order in which we consider the spaces, 
is dense in the 
leftmost component of the triplet, it follows that $\cD$ is dense in $\cH_-$ as well.

In order to complete the proof, by Lemma~\ref{l:be} it is sufficient to prove that 
there exists a contractive and boundedly invertible operator 
$B\colon\cH_-\ra\cH_+$ 
such that \eqref{e:beta} holds. Indeed, for arbitrary $\phi\in\dom(j_+^*)=\dom(j_-)$ 
and $u\in\dom(j_+)$ we have
\begin{equation}\label{e:lap} \langle \phi,u\rangle_{\cH_0}  = \langle \phi,j_+u\rangle_{\cH_0}
=\langle j_+^* \phi,u\rangle_{\cH_+}  = \langle V\phi,u\rangle_{\cH_+},
\end{equation} 
where $V=j_+^*$ but considered as an operator defined in $\cH_-$ and 
valued in $\cH_+$. By Theorem~\ref{t:triplet}, there 
exists a unique unitary operator $\widetilde V\colon \cH_-\ra\cH_+$ that 
extends the operator $V$. Letting $B=\widetilde V$, clearly, since $B$ is 
unitary it is contractive and boundedly invertible.
In particular, \eqref{e:lap} can be rewritten as
\begin{equation}\label{e:lup} 
\langle \phi,u\rangle_{\cH_0}=\langle B\phi,u\rangle_{\cH_+},\quad \phi\in
\dom(j_-),\ u\in\dom(j_+).
\end{equation} 

We assume now that the condition (a) holds and prove that 
\eqref{e:lup} holds for all $\phi,u\in\cD$, that is, \eqref{e:beta}. To this end, fix 
$\phi\in\dom(j_-)$ for the moment and consider an arbitrary vector $u\in\cD$. 
As proven before, $\dom(j_+^*j_+)$ lies in $\cD$ and is dense with respect
to the norm $\|\cdot\|_{\cH_+}$ hence, there exists a sequence $(u_n)$ 
of vectors in $\dom(j^*_+j_+)$ such that $\|u_n-u\|_{\cH_+}\ra 0$ as 
$n\ra\infty$. 
By condition (a), $\langle \phi,u_n\rangle_{\cH_0}\ra \langle \phi,u
\rangle_{\cH_0}$ as $n\ra\infty$, and by \eqref{e:lup} we have
\begin{equation*} \langle \phi,u_n\rangle_{\cH_0}=\langle B\phi,u_n
\rangle_{\cH_+},\quad n\in\NN,
\end{equation*} hence, we can pass to the limit as $n\ra\infty$ to 
obtain that 
\eqref{e:lup} holds for all $u\in\cD$ and all $\phi\in\dom(j_-)$. Next, a similar
reasoning, with fixing $u\in\cD$ and approximating $\phi\in\cD$ accordingly, 
shows that \eqref{e:lup} holds for all $u,\phi\in\cD$. 

It is clear that (c)$\Ra$(b)$\Ra$(a). Observe that, we have just proven before 
that (a)$\Ra$(c), hence the conditions (a), (b), and (c) are mutually 
equivalent.

The converse implication is clear.
\end{proof}

\begin{example} \emph{Weighted $L^2$ Spaces.} \label{ex:wledoi}
Let $(X;\fA)$ be a measurable space on which we consider a $\sigma$-finite
measure $\mu$.  A function 
$\omega$ defined on $X$ is called a \emph{weight} with respect to the 
measure space 
$(X;\fA;\mu)$ if it is measurable and $0<\omega(x)<\infty$, for $\mu$-almost all 
$x\in X$. Note that $\cW(X;\mu)$, the collection of weights with respect to 
$(X;\fA;\mu)$, is a multiplicative unital group. For an arbitrary
$\omega\in \cW(X;\mu)$, consider the measure $\nu$ whose Radon-Nikodym 
derivative with respect to $\mu$ is $\omega$, denoted $\de\nu=\omega\de\mu$,
that is, for any $E\in\fA$ we have
$\nu(E)=\int_E \omega\de\mu$.
It is easy to seee, e.g.\ see \cite{CojGh4}, that
 $\nu$ is always $\sigma$-finite. 
 
 In \cite{CojGh4}, Theorem~2.1, it is proven that
$(L_w^2(X;\mu);L^2(X;\mu);L^2_{w^{-1}}(X;\mu))$ is a triplet of 
closely embedded Hilbert spaces, provided that $w$ is a weight 
on the $\sigma$-finite measure space $(X;\mathfrak{A};\mu)$. More precisely,
the closed embeddings $j_\pm$ of $L_w^2(X;\mu)$ in $L^2(X;\mu)$ and
of $L^2(X;\mu)$ in 
$L^2_{w^{-1}}(X;\mu)$ have maximal domains $L^2_\omega(X;\mu)\cap 
L^2(X;\mu)$ and, respectively, $L^2(X;\mu)\cap L^2_{\omega^{-1}}(X;\mu)$. 
It is a routine exercise to check that, if $\phi,u\in L^2(X;\mu)\cap L^2_w(X;\mu)\cap L^2_{w^{-1}}(X,\mu)$ then
\begin{equation*}|\langle\phi,u\rangle_{L^2(X;\mu)}|\leq \|\phi
\|_{L^2_w(X;\mu)} \|u\|_{L^2_{w^{-1}}(X;\mu)},
\end{equation*} hence, $(L_w^2(X;\mu);L^2(X;\mu);L^2_{w^{-1}}(X;\mu))$
is a generalised triplet of Hilbert spaces as well, by Theorem~\ref{t:genclo}.
Observe that, in the proof of Theorem~2.1 in \cite{CojGh4}, it was directly
proven that $L_w^2(X;\mu)\cap L^2(X;\mu)\cap L^2_{w^{-1}}(X;\mu)$ 
is dense in each of the spaces $L_w^2(X;\mu)$, $L^2(X;\mu)$, and
$L^2_{w^{-1}}(X;\mu)$.
\end{example}

\begin{example} \emph{Dirichlet Type Spaces.} \label{ex:dirichlet}
For a fixed natural number $N$ consider
the unit polydisc $\DD^N=\DD\times \cdots\times\DD$, the direct product of $N$ 
copies 
of the unit disc $\DD=\{z\in\CC\mid |z|<1\}$. We consider $H(\DD^N)$  the algebra 
of functions holomorphic in the 
polydisc, that is, the collection of all functions $f\colon \DD^N\ra \CC$ that are 
holomorphic in each variable, equivalently, there exists $(a_k)_{k\in\ZZ_+^N}$ 
with the property that
\begin{equation}\label{e:rep} f(z)=\sum_{k\in\ZZ_+^N} a_kz^k,\quad z\in\DD^N,
\end{equation} where the series converges absolutely and uniformly on any 
compact subset in $\DD^N$. Here and in the sequel, for any multi-index 
$k=(k_1,\ldots,k_N)\in\ZZ_+^N$ and any $z=(z_1,\ldots,z_N)\in \CC^N$ we let 
$z^k=z_1^{k_1}\cdots z_N^{k_N}$.  

Let $\alpha\in\RR^N$ be fixed.
The \emph{Dirichlet type space} $\cD_\alpha$, see \cite{Taylor} and \cite{JupiterRedett}, 
is defined as the space of all functions $f\in H(\DD^N)$ with representation \eqref{e:rep} 
subject to the condition
\begin{equation*}\label{e:cond} \sum_{k\in\ZZ_+^N} (k+1)^\alpha |a_k|^2<\infty,
\end{equation*}
where, $(k+1)^\alpha=(k_1+1)^{\alpha_1}
\cdots(k_N+1)^{\alpha_N}$.
The linear space 
$\cD_\alpha$ is naturally organized as a Hilbert space with inner product 
$\langle\cdot,\cdot\rangle_\alpha$
\begin{equation*}\label{e:ip} \langle f,g\rangle_\alpha
=\sum_{k\in\ZZ_+^N} (k+1)^\alpha 
a_k \overline{b_k},
\end{equation*} where $f$ has representation \eqref{e:rep} and similarly
$g(z)=\sum_{k\in\ZZ_+^N} b_k z^k$, for all $z\in\DD^N$, 
and norm $\|\cdot\|_\alpha$ defined by
\begin{equation*}\label{e:norm} \|f\|^2_\alpha=
\sum_{k\in\ZZ_+^N} (k+1)^\alpha 
|a_k|^2.\end{equation*}

It is proven in \cite{CojGh4},
Theorem~3.1, that, letting $\alpha,\beta\in\RR^N$ be arbitrary multi-indices, 
then
$(\cD_\beta;\cD_\alpha;\cD_{2\alpha-\beta})$ is a triplet of closely embedded 
Hilbert spaces. It is a simple exercise to check that
\begin{equation*} |\langle f,g\rangle_\alpha|\leq \|f\|_{2\alpha-\beta} 
\|g\|_\beta,
\end{equation*} whenever 
$f,g\in\cD_\beta\cap\cD_\alpha\cap\cD_{2\alpha-\beta}$ 
hence, by Theorem~\ref{t:genclo}
$(\cD_\beta;\cD_\alpha;\cD_{2\alpha-\beta})$
is a generalised triplet of Hilbert spaces as well. Note that, 
in this particular case, 
$\cD_\beta\cap\cD_\alpha\cap\cD_{2\alpha-\beta}$ contains 
$\cP_N$, the linear space of polynomial functions in $N$ complex variables, 
that is dense in each of the Dirichlet type spaces
$\cD_\beta$, $\cD_\alpha$, and $\cD_{2\alpha-\beta}$.
\end{example}

\subsection{Starting with a Generalised Triplet.}
We now consider a generalised triplet of Hilbert spaces $(\cH;\cH_0;\cH^
\prime)$. First, we investigate the possibility of making a triplet of closely 
embedded Hilbert spaces out of it, in a natural way, and in such a way 
that it "essentially" coincides with it. 
Let $\cD=\cH\cap\cH_0\cap\cH^\prime$ be the linear subspace 
that is dense in each of $\cH$, $\cH_0$, and $\cH^\prime$ and, in view of 
Lemma~\ref{l:be}, consider the contractive linear operator 
$B\colon \cH^\prime\ra\cH$ that is boundedly 
invertible and such that \eqref{e:beta} holds. 

Let $j_{+,0}$ be the linear operator with domain $\dom(j_{+,0})=\cD$, 
considered as a 
subspace of $\cH$, as the embedding in $\cH_0$, that is, $j_{+,0}u=u$ for all 
$u\in\cD$. We observe that, for any $u,\phi\in\cD$, we have
\begin{equation}\label{e:jestar}
\langle \phi,j_{+,0}u\rangle_{\cH_0}=\langle \phi,u\rangle_{\cH_0}=
\langle B\phi,u\rangle_\cH,
\end{equation} hence $\cD\subseteq \dom(j_{+,0}^*)$ and $B|\cD=j_{+,0}^*|\cD$.
In particular, $j_{+,0}^*$ is defined on a subspace dense in $\cH_0$, hence 
$j_{+,0}$ is closable. Let $T\in\cC(\cH_0,\cH)$ 
be the closure of the operator $j_{+,0}$. Thus, $\cD\subseteq \dom(T)$,
$Tu=u$ for all $u\in\cD$, and $j_{+,0}^*=T^*$, in particular, $T^*|\cD=B|\cD$. 
Therefore,
\begin{equation*} \ker(T)=\cH\ominus\overline{\ran(T^*)}\subseteq \cH\ominus 
\overline{B\cD}=\cH\ominus \cH=0,
\end{equation*}
where we have taken into account that $B$ is boundedly invertible and $\cD$ is 
dense in $\cH^\prime$. This shows that $T$ is one-to-one. 
Since $T\cD\supseteq j_{+,0}\cD=\cD$, which is dense in $\cH_0$ as well, it
follows that $T$ has dense range in $\cH_0$. 

We can now consider the triplet of closely embedded Hilbert spaces 
$(\cR(T);\cH_0;\cD(T^*))$, where $j_T$ is the closed embedding of 
$\cR(T)$ in $\cH_0$ and $i_{T^*}^{-1}$ is the closed embedding of 
$\cH_0$ in 
$\cD(T^*)$, see the definitions and properties in Section~\ref{s:tcehs}. 
In the following we show that the subspace $\cD$ is 
densely contained in $\cR(T)$, that on $\cD$ 
the inner product of $\cH$ coincides with that of $\cR(T)$,
and that on $\cD$ the topological structure of $\cH^\prime$ 
coincides with that of $\cD(T^*)$.

Since $\cD$ is a subspace of $\dom(T)$ and $T$ acts on $\cD$ like the 
identity operator, it follows that $\cD$ is a subspace of $\ran(T)$, hence a 
subspace of $\cR(T)$. In addition, taking into account that $T$ is the closure of the 
embedding operator $j_{+,0}$, for any vector $x\in\dom(T)$ there exists a sequence 
$(x_n)_n$ of vectors in $\cD$ such that $\|x-x_n\|_\cH\ra 0$ and 
$\|j_{+,0}x_n-Tx\|_{\cH_0}\ra 0$, as $n\ra\infty$. Hence
\begin{equation*} \|x_n-Tx\|_T=\|x_n-x\|_\cH\ra 0,\mbox{ as }n\ra \infty,
\end{equation*} which shows that $\cD$ is dense in $\ran(T)$ with respect to the 
norm $\|\cdot\|_T$, see \eqref{e:uvete}. Since $\ran(T)$ is dense in $\cR(T)$ with 
respect to the norm $\|\cdot\|_T$, it follows that $\cD$ is dense in $\cR(T)$.
On the other hand, for any vector $x\in\cD$ we have
\begin{equation*} \|x\|_T=\|Tx\|_T=\|x\|_\cH,
\end{equation*} that is, on $\cD$ the norms of the two Hilbert spaces $\cH$ and 
$\cR(T)$ coincide.

On the other hand, as a consequence of \eqref{e:jestar} and taking into account 
that $T^*=j_{+,0}^*$, it follows that $\cD\subseteq\dom(T^*)$ and $T^*|\cD=B|\cD$. 
Thus, $\cD$ is a subspace of $\cD(T^*)$ and
\begin{equation} |x|_{T^*}=\|T^*x\|_\cH=\|Bx\|_\cH,\quad x\in\cD.
\end{equation} Taking into account that $B\colon\cH^\prime\ra\cH$ is bounded 
and boundedly invertible, this implies that on $\cD$ the norms $|\cdot|_{T^*}$ and 
$\|\cdot\|_{\cH^\prime}$ are equivalent.
%

We have proven the following
\begin{theorem}\label{t:gth} Let $(\cH;\cH_0;\cH^\prime)$ be a generalised triplet
of Hilbert spaces, let $\cD=\cH\cap\cH_0\cap\cH^\prime$ be the linear subspace 
which is dense in each of $\cH$, $\cH_0$, and $\cH^\prime$, 
and let $B\colon\cH^\prime\ra\cH$ denote the contractive linear 
operator that is boundedly invertible and such that \eqref{e:beta} holds.

\nr{1} Let $j_{+,0}$ denote the linear operator with domain $\dom(j_{+,0})=\cD$, 
considered as a subspace of $\cH$, and with codomain in $\cH_0$, defined by 
$j_{+,0}u=u$ for all $u\in\cD$. Then $j_{+,0}$ is closable.

\nr{2} Let $T$ denote the closure of $j_{+,0}$. Then $T\in\cC(\cH,\cH_0)$, 
is one-to-one, has dense range, and $T^*|\cD=B|\cD$.

\nr{3} The triplet of closely embedded Hilbert spaces 
$(\cR(T);\cH_0;\cD(T^*))$, where $j_T$ is the closed embedding of 
$\cR(T)$ in $\cH_0$ and $i_{T^*}^{-1}$ is the closed embedding of $\cH_0$ in 
$\cD(T^*)$, has the following properties: 
\begin{itemize}
\item[(i)] $\cD$ is 
densely contained in $\cR(T)$ and on $\cD$ 
the inner product of $\cH$ coincides with that of $\cR(T)$;
\item[(ii)] $\cD$ is a subspace of $\cD(T^*)$
and on $\cD$ the norms $\|\cdot\|_{\cH^\prime}$ and $|\cdot|_{T^*}$ are equivalent.
\end{itemize}
\end{theorem}

The triplet of closely embedded Hilbert spaces $(\cR(T);\cH_0;\cD(T^*))$
constructed out of the generalised triplet of Hilbert spaces 
$(\cH;\cH_0;\cH^\prime)$ as in Theorem~\ref{t:gth}, 
"essentially" coincides with the triplet of generalised Hilbert spaces 
$(\cH;\cH_0;\cH^\prime)$ on the linear manifold $\cD$, 
that is dense in each of the spaces $\cH$, $\cH_0$, and $\cH^\prime$, 
modulo a norm equivalent with $\|\cdot\|_{\cH^\prime}$. 
If we want the generalised triplet $(\cH;\cH_0;\cH^\prime)$ be a triplet 
of closely embedded Hilbert spaces itself, this depends on a rather general 
question of when a closed embedding can be obtained from an unbounded 
embedding by taking its closure. We record this fact in the following

\begin{remark}\label{r:ce} Let $\cH$ and $\cG$ be two Hilbert spaces such 
that there exists $\cD_0$ a linear manifold of both $\cH$ and $\cG$ 
that is dense in $\cG$ and let the embedding operator 
$j_0\colon\cD_0\ra \cH$ be
defined by $j_0u=u$ for all $u\in\cD_0$.
Then,  the following assertions are equivalent:
\begin{itemize}
\item[(a)] $j_0$ is closable, as an operator defined in $\cG$ and valued in 
$\cH$, 
and the closure $j=\overline{j_0}$ is a closed embedding of $\cG$ in $\cH$, 
in the sense of the definition as in Subsection~\ref{ss:dt}. 
\item[(b)] For every sequence 
$(u_n)$ of vectors in $\cD_0$ that is Cauchy with respect to both norms 
$\|\cdot\|_\cH$ and $\|\cdot\|_\cG$, there exists $u\in\cH\cap\cG$ 
(of course, unique) such that
$\|u_n-u\|_\cH\ra 0$ and $\|u_n-u\|_\cG\ra 0$ as $n\ra\infty$.
\end{itemize}
\end{remark}

We can now approach the main question of this section referring to 
characterisations of those generalised triplets of Hilbert spaces that are also 
triplets of closely embedded Hilbert spaces.

\begin{theorem}\label{t:gtceh}
Let $(\cH;\cH_0;\cH^\prime)$ be a generalised triplet
of Hilbert spaces, let $\cD=\cH\cap\cH_0\cap\cH^\prime$ be
the linear subspace which is dense in each of $\cH$, $\cH_0$, and 
$\cH^\prime$, 
and let $B\colon\cH^\prime\ra\cH$ denote the contractive linear 
operator that is boundedly invertible and such that \eqref{e:beta} holds.
Then, $(\cH;\cH_0;\cH^\prime)$ is a triplet of closely embedded Hilbert 
spaces, modulo a renorming of $\cH^\prime$ with an equivalent norm, 
if and only if the following three conditions hold:
\begin{itemize}
\item[(i)] For any sequence $(u_n)$ of vectors in $\cD$ that is Cauchy 
with respect to both norms $\|\cdot\|_\cH$ and $\|\cdot\|_{\cH_0}$, 
there exists $u\in\cH\cap\cH_0$ such that $\|u_n-u\|_\cH\ra0$ and 
$\|u_n-u\|_{\cH_0}\ra 0$ as $n\ra\infty$.
\item[(ii)] For any sequence $(\phi_n)$ of vectors in $\cD$ that is Cauchy 
with respect to both norms $\|\cdot\|_{\cH^\prime}$ and $\|\cdot\|_{\cH_0}$, 
there exists $\phi\in\cH^\prime\cap\cH_0$ such that 
$\|\phi_n-\phi\|_{\cH^\prime}\ra0$ and 
$\|\phi_n-\phi\|_{\cH_0}\ra 0$ as $n\ra\infty$.
\item[(iii)] For every vector $\phi\in\cH_0$ with the property that the linear 
functional $\cD\ni u\mapsto \langle y,\phi\rangle_{\cH_0}$ is bounded with respect 
to the norm $\|\cdot\|_\cH$, there exists a sequence $(\phi_n)_n$ of vectors in
in $\cD$ such that $\|\phi_n-\phi\|_{\cH_0}\ra 0$ and 
$\|\phi_n-\phi\|_{\cH^\prime}$, as $n\ra\infty$.
\end{itemize}
\end{theorem}

\begin{proof} We first assume that the generalised triplet of Hilbert spaces 
$(\cH;\cH_0;\cH^\prime)$ satisfies  all conditions (i)--(iii). Consider 
the operator $j_{+,0}$ with domain $\dom(j_{+,0})=\cD$, 
viewed as a subspace of $\cH$, as the embedding in $\cH_0$, that is, 
$j_{+,0}u=u$ for all $u\in\cD$. By \eqref{e:jestar}, 
$\cD\subseteq\dom(j_{+,0}^*)$ and 
hence $j_{+,0}$ is closable. By condition (i), see Remark~\ref{r:ce}, it 
follows that the closure $j_+$ of $j_{+,0}$ is an embedding, that is, for all 
$u\in\dom(j_+)$ we have $j_+u=u$. With notation as in 
Theorem~\ref{t:gth}, this means that $T=j_+$ is a closed embedding of 
$\cH$ in $\cH_0$. In addition, by condition (i) it also follows that
\begin{equation}\label{e:jebe} \langle \phi,u\rangle_{\cH_0}=\langle B\phi,u
\rangle_\cH,\quad u\in\dom(j_+),\ \phi\in\cD.
\end{equation}

Further on, by changing the norm $\|\cdot\|_{\cH^\prime}$ with an equivalent 
norm, without loss of generality we can assume that the operator 
$B\colon\cH^\prime\ra\cH$ is unitary. We consider the operator $i_{-,0}$ 
with domain $\cD$ considered as a subspace of $\cH^\prime$ and range in 
$\cH_0$, defined by $i_{-,0}\phi=\phi$ for all $\phi\in\cD$. We observe 
that, for any $u,\phi\in\cD$ we have
\begin{equation}\label{e:iestar} \langle i_{-,0}\phi,u\rangle_{\cH_0}=
\langle \phi,u\rangle_{\cH_0}=\langle B\phi,u\rangle_{\cH}=\langle \phi,B^* u
\rangle_{\cH^\prime},
\end{equation} hence $\cD\subseteq\dom(i_{-,0}^*)$ and $B^*|\cD=i_{-,0}^*$. 
Therefore, $i_{-,0}^*$ is defined on a subspace dense in $\cH_0$ and hence
$i_{-,0}$ is closable. From condition (ii) it follows that $i_-$, the closure of the 
operator $i_{-,0}$, is an embedding, that is, $\dom(i_-)\subseteq \cH_0\cap
\cH^\prime$ and $i_-\phi=\phi$ for all $\phi\in\dom(i_-)$. 
Clearly, $\dom(i_-)=\ran(i_-)$ is dense in both $\cH_0$ and $\cH^\prime$, 
hence we can consider $j_-=i_-^{-1}$, which is a closed embedding of 
$\cH_0$ in $\cH^\prime$ with dense range. In addition, by condition (ii) 
and \eqref{e:jebe} it follows that
\begin{equation}\label{e:jebed} \langle \phi,u\rangle_{\cH_0}=\langle B\phi,u
\rangle_\cH,\quad u\in\dom(j_+),\ \phi\in\dom(j_-).
\end{equation}

So far, we have shown that the triplet $(\cH;\cH_0;\cH^\prime)$ 
satisfies the axioms (th1) and (th2), with respect to the closed embeddings 
$j_+$ and $j_-$ defined as before. Recalling that $B$ is unitary, from 
\eqref{e:jebed} it follows that, for every $\phi\in \dom(j_-)$, we have
\begin{equation*} \sup\bigl\{\frac{|\langle \phi,u\rangle_{\cH_0}|}{\|u\|_+}\mid 
u\in\dom(j_+),\ u\neq 0\bigr\}=\|B\phi\|_\cH=\|\phi\|_{\cH^\prime},
\end{equation*}
hence \eqref{e:dual} holds. It only remains to prove that 
$\dom(j_+^*)\subseteq\dom(j_-)$. To this end, let $\phi\in\dom(j_+^*)$, hence,
the linear functional $\cD\ni u\mapsto \langle j_+u,\phi\rangle_{\cH_0}
=\langle u,\phi\rangle_{\cH_0}$ is continuous with respect to the norm 
$\|\cdot\|_{\cH}$. By condition (iii), there exist a sequence
$(\phi_n)_n$ of vectors in $\cD$ such that $\|\phi_n-\phi\|_{\cH_0}\ra 0$ and 
$\|\phi_n-\phi\|_{\cH^\prime}$, as $n\ra\infty$, which means that 
$\phi\in\dom(j_-)$.
\end{proof}

\end{document}